\documentclass[12pt,halfline,a4paper]{ouparticle}

\usepackage{float}
\usepackage{indentfirst}
\usepackage{amsmath, amsthm, amssymb, amsfonts}
\usepackage{mathtools}
\usepackage{physics}
\usepackage{graphicx}
\usepackage{verbatim}
\usepackage[hidelinks]{hyperref}
\usepackage{color,algorithm,algorithmic}
\usepackage[nottoc]{tocbibind}
\usepackage{enumerate}
\usepackage{caption}
\usepackage{subcaption}
\usepackage{comment}

\newtheorem{theorem}{Theorem}
\newtheorem{lem}{Lemma}
\newtheorem{corollary}{Corollary}
\newtheorem{remark}[theorem]{Remark} 

\newcommand{\RNum}[1]{\uppercase\expandafter{\romannumeral #1\relax}}

\newenvironment{customproof}
  {\par\noindent\textit{Proof of Theorem~\ref{main_theorem}.}\par\noindent}
  {\hfill$\square$\par}
\begin{document}

\title{Note on the positivity of the real part of the log-derivative of the Riemann $\xi$-function near the critical line}

\author{%
\name{Andrius Grigutis, Lukas Turčinskas}
\address{Institute of Mathematics, Vilnius University,\\
Naugarduko g. 24, LT-03225, Vilnius, Lithuania}
\email{andrius.grigutis@mif.vu.lt, lukas.turcinskas@mif.stud.vu.lt}
}

\abstract{In this work, we investigate the positivity of the real part of the log-derivative of the Riemann $\xi$-function in the region $1/2+1/\sqrt{\log t}<\sigma<1$, where $t$ is sufficiently large. We provide an explicit lower bound for $\mathfrak{R}\sum_{\rho}1/(s-\rho)$, where the summation runs over the zeta-zeros on the critical line. We also consider hypothetical cases of positivity of the log-derivative of the Riemann $\xi$-function in the provided region, assuming that there are non-trivial zeta-zeros off the critical line.}

\date{\today}

\keywords{Riemann $\xi$-function; Riemann $\zeta$-function; logarithmic derivative; positivity}

\maketitle

\section{Introduction and main result}\label{sec1}
Let $s=\sigma+it$ be the complex variable. The Riemann $\xi$-function is given as
\[
\xi(s) = \frac{1}{2} s (s-1) \pi^{-s/2} \Gamma(s/2) \zeta(s),
\]
where $\zeta(s)$ is the Riemann $\zeta$-function. Both Riemann $\zeta$ and $\xi$ functions have the same zeros in the strip $0<\sigma<1$, called the critical strip, and the Riemann hypothesis (RH) states that these zeros all lie on the line $1/2+it,\,t\in\mathbb{R}$, called the critical line. These days, it is unknown if the RH is true. RH is numerically confirmed up to $t\leqslant3\times10^{12}$, see Platt and Trudgian \cite{Platt}.
The function $\xi(s)$ satisfies functional equation $\xi(s)=\xi(1-s)$ and for its conjugate it is true that $\overline{\xi(s)} = \xi(\overline{s})$. Therefore if $s$ is a zero of $\xi(s)$ then so are the numbers $\overline{s}$, $1-s$, and $1-\overline{s}$. 

Let $\rho=\beta+i\gamma$ be non-trivial zero of $\zeta(s)$. Let $N(T)$ denote the number of zeta-zeros in the rectangle $0<\sigma<1$, $0<t< T$, and $N_{1/2}(T)$ denote the number of zeta-zeros on the critical line $1/2+it$, $0<t<T$. If we denote
\begin{align}\label{c_definition}
    c\coloneqq\liminf_{T\to\infty}\frac{N_{1/2}(T)}{N(T)},
\end{align}
then it follows that
\begin{equation}\label{c_inequality}
    cN(T)\leqslant N_{1/2}(T)\leqslant N(T),
\end{equation}
for all $T\geqslant 0$.
In the middle of the last century, it was proved (see Selberg \cite{Selberg}) that $c>0$. From there, this estimate was improved multiple times, see for example Levinson \cite{Levinson} ($c>1/3$), Conrey \cite{Conrey} ($c>2/5$), Pratt, Robles, Zaharescu, and Zeindler \cite{Pratt} ($c>5/12$).

The function $\xi(s)$ can be expanded as an infinite product over $\rho$ (see, for example, Edwards \cite{Edwards})
\begin{equation}\label{xi_infinite_product}
    \xi(s) = \xi(0) \prod_{\rho} \left( 1 - \frac{s}{\rho} \right) = \frac{1}{2} \prod_{\rho} \left( 1 - \frac{s}{\rho} \right),
\end{equation}
where the product in \eqref{xi_infinite_product} is taken in an order that pairs each root $\rho$ of $\xi(s)$ with the corresponding one $1-\rho$. From \eqref{xi_infinite_product} it is easy to see that for $s$ such that $\xi(s)\neq0$ the logarithmic derivative of $\xi(s)$ is
\begin{equation}\label{log_derivative_xi}
    \frac{\xi'}{\xi}(s) = \sum_{\rho} \frac{1}{s - \rho},
\end{equation}
where the order of summation in \eqref{log_derivative_xi} is understood the same way as the order of product in \eqref{xi_infinite_product}.

It is known that (see Hinkkanen \cite{Hinkkanen})
\begin{equation*}
    \mathfrak{R} \frac{\xi'}{\xi}(s) > 0, \quad \text{when}\;\; \mathfrak{R}(s)>1,
\end{equation*}
and RH is equivalent to
\[
\mathfrak{R} \frac{\xi'}{\xi}(s) > 0, \quad \text{when}\;\; \mathfrak{R}(s)>\frac{1}{2}.
\]
Lagarias \cite{Lagarias} proved that
\begin{equation}\label{inf_log_derivative}
    \inf \left\{ \mathfrak{R} \frac{\xi'}{\xi}(s) : -\infty < t < \infty \right\} = \frac{\xi'}{\xi}(\sigma),
\end{equation}
for $\sigma >10$ and Garunkštis \cite{Garunkstis} later improved \eqref{inf_log_derivative} for $\sigma >a$, where $\sigma>a$ is zero-free region of $\zeta(s)$.
The following facts are also known about Riemann $\xi$-function.

\begin{theorem}
    (Sondow, Dumitrescu \cite{Sondow}) The following statements are equivalent.
    \begin{enumerate}[I.]
        \item If $t$ is any fixed real number then $\abs{\xi(\sigma+it)}$ is increasing for $1/2<\sigma<\infty$.
        \item If $t$ is any fixed real number then $\abs{\xi(\sigma+it)}$ is decreasing for $-\infty<\sigma<1/2$.
        \item The Riemann hypothesis is true.
    \end{enumerate}
\end{theorem}

\begin{theorem}
    (Sondow, Dumitrescu \cite{Sondow}) The $\xi$-function is increasing in modulus along every horizontal half-line lying in any open right half-plane that contains no $\xi$ zeros. Similarly, the modulus decreases on each horizontal half-line in any zero-free, open left half-plane.
\end{theorem}

\begin{theorem}
    (Matiyasevich, Saidak, Zvengrowski \cite{Matiyasevich})
    Let $\sigma_0$ be greater than or equal to the real part of any zero of $\xi$. Then $\abs{\xi(s)}$ is strictly increasing with respect to $\sigma$ in the half-plane $\sigma>\sigma_0$.
\end{theorem}
\noindent In light of these results on $\abs{\xi(s)}$ behavior and the fact that $\xi(s)$ is holomorphic in the region $0<\sigma<1$ we can use the following lemma.

\begin{lem}\label{real_log_derivative}
(a) Let $f$ be holomorphic in an open domain $D$ and not identically zero. Let us also suppose $\mathfrak{R} f'/f(s) < 0$ for all $s \in D$ such that $f(s) \neq 0$. Then $|f(s)|$ is strictly decreasing with respect to $\sigma$ in $D$, i.e., for each $s_0 \in D$ there exists a $\delta > 0$ such that $|f(s)|$ is strictly monotonically decreasing with respect to $\sigma$ on the horizontal interval from $s_0 - \delta$ to $s_0 + \delta$.

(b) Conversely, if $|f(s)|$ is decreasing with respect to $\sigma$ in $D$, then $\mathfrak{R} f'/f(s) \leqslant 0$ for all $s \in D$ such that $f(s) \neq 0$.
\end{lem}

\begin{proof}
    See Matiyasevich, Saidak, Zvengrowski \cite{Matiyasevich}.
\end{proof}

\begin{remark}
Analogous results hold for monotone increasing $\abs{f(s)}$ and $\mathfrak{R} f'/f(s) > 0$.
\end{remark}

\noindent Applying Lemma \ref{real_log_derivative} to $\xi(s)$ we see that $\mathfrak{R}\xi'/\xi(s)>0$ with $\sigma>\tilde{\beta}$, where $\sigma>\tilde{\beta}$ is zero-free region of $\zeta(s)$. Besides the mentioned results, the positivity of $\mathfrak{R}\xi'/\xi(s)$ in the strip $1/2<\sigma<1$ was studied in \cite{Grigutis}. An explicit lower and upper bounds of $\mathfrak{R}\sum_{\rho}1/(s-\rho)$, by letting the summation run over the zeta-zeros on the critical line, were given there in the rectangle $1/2<\sigma<1$ and $t$ greater than some sufficiently large value. More precisely, the following Theorem was proven.

\begin{theorem}\label{Goldstein, Grigutis theorem}
(Goldštein and A.G. \cite{Grigutis}) Let $1/2<\sigma<1$ and $0<c\leqslant 1$ be such that $c\leqslant N_{1/2}(T)/N(T)$ for all $T\geqslant \gamma_1=14.134725\ldots$, where $\gamma_1$ is the lowest non-trivial zeta-zero $\zeta(1/2+i\gamma_1)=0$. Let
\begin{equation*}
\begin{split}
A(t) &= 0.12\log{\frac{t}{2\pi}}-2.32\log{\log{t}}-18.432-\varepsilon_1(t), \\
B(t) &= 0.49\log{\frac{t}{2\pi}}+0.58\log{\log{t}}-4.603+\varepsilon_2(t),
\end{split}
\end{equation*}
where $\varepsilon_1(t)$ and $\varepsilon_2(t)$ are known explicit $t$ functions both vanishing as $t^{-1}\log{t}$, $t\to\infty$, \cite[eqs. (15), (16)]{Grigutis}.
Then
\begin{align}
0 < c \left( \sigma - \frac{1}{2} \right) A(t) < &\;\mathfrak{R} \sum_{\rho = 1/2 + i\gamma} \frac{1}{s - \rho}, \quad t > 1.984 \times 10^{114}, \label{LB}\\
&\;\mathfrak{R} \sum_{\rho = 1/2 + i\gamma} \frac{1}{s - \rho} < \frac{B(t)}{\sigma - 1/2}, \quad t > 14.635. \label{UB}
\end{align}
\end{theorem}

The lower bound \eqref{LB} allows asking the following question: can
\begin{align*}
\mathfrak{R}\frac{\xi'}{\xi}(s)=\mathfrak{R}\sum_{\rho = 1/2 + i\gamma} \frac{1}{s - \rho}
+\mathfrak{R}\sum_{\tilde{\rho} = \tilde{\beta} + i\tilde{\gamma}} \frac{1}{s - \tilde{\rho}}=:\Sigma_1+\Sigma_2
\end{align*}
be positive in the critical strip if $\Sigma_2$ exists? Here $\tilde{\rho}$ denotes the hypothetical zeta-zero off the critical line. Such types of questions were also considered in \cite{Grigutis}.

It is easy to see that the lower bound in \eqref{LB} tends to zero as $\sigma\to1/2$. Thus, the goal of this paper is to improve the result \eqref{LB}, obtaining a lower bound which grows infinitely when $\sigma\to1/2$. This requires the step off of the size $1/\sqrt{\log t},\,t\to\infty$ from the critical line $1/2+it,\,t\in\mathbb{R}$. More precisely, the main result of this article consists of the following theorem.

\begin{theorem}\label{main_theorem}
Let $1/2+1/\sqrt{\log t}<\sigma<1$ and $0<c\leqslant 1$ be such that $c\leqslant N_{1/2}(T)/N(T)$ for all $T\geqslant \gamma_1=14.134725\ldots$, where $\gamma_1$ is the lowest non-trivial zeta-zero $\zeta(1/2+i\gamma_1)=0$. Then
\begin{equation}\label{main_lower_bound}
0 < \frac{(0.28-\varepsilon(t))c}{\sigma-1/2} < \mathfrak{R} \sum_{\rho = 1/2 + i\gamma} \frac{1}{s - \rho}, \quad t > 3.11 \times 10^{10},
\end{equation}
where $\varepsilon(t)$ is the known function, vanishing as $\log\log t/\log t$, $t\to\infty$, see \eqref{epsilon}.
\end{theorem}

We prove Theorem \ref{main_theorem} in Section \ref{sec3}. In Figure \ref{main_thm_figure}, we depict the area given by the simplified version of the inequality \eqref{main_lower_bound}. That is, we set $c=1$, $\varepsilon(t)=0$, and replace an infinite sum over zeros with the finite sum. Then, in the $\sigma\times t$ plane, we plot the area of complex numbers $s=\sigma+it$ that satisfy the inequality
\[
\frac{0.28}{\sigma-1/2} < \mathfrak{R} \sum_{k=0}^{9} \frac{1}{s - \rho_{j+k}},
\]
where $\rho=\rho_{j+k}$ denote the $j+k$-th zeta-zero on the critical line; here $j=10^{12}+1$. The imaginary parts of $\rho_{j+k}$ are taken from the Odlyzko database \cite{Odlyzko}.  

\begin{figure}[H]
    \centering
    \includegraphics[scale=0.7]{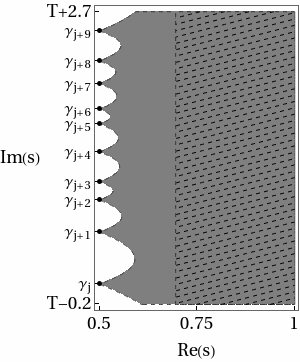}
    \caption{The entire gray region shows where the inequality $0.28/(\sigma - 1/2) < \mathfrak{R} \sum_{k=0}^{9} \frac{1}{s - \rho_{j+k}}$ holds. The dashed gray region denotes $1/2+1/\sqrt{\log t}<\sigma<1$, $T-0.2<t<T+2.7$, $T=\gamma_j\approx 2.68\times 10^{11}$, $j=10^{12}+1$. Depicted with \cite{Mathematica}.}
    \label{main_thm_figure}
\end{figure}

As in paper \cite{Grigutis}, in Section \ref{sec4} of this work, we analyze different scenarios in which we assume that the Riemann hypothesis fails. Under such an assumption, we show that $\mathfrak{R} \xi'/\xi(s)$ may remain positive near the critical line except for some small regions near hypothetical non-trivial zeta-zeros off the critical line.

\section{Lemmas}\label{sec2}

In this Section, we formulate and prove several auxiliary statements that are needed in proving the main result, Theorem \ref{main_theorem}. One of the main tools used in this paper is the so-called partial summation.
\begin{lem}\label{partial_summation}
Let $\{a_n\}_{n=1}^{\infty}$ be a sequence of complex numbers and $f(t)$ a continuously differentiable function on $[1, x]$. If $A(t) = \sum_{n\leqslant t} a_n$, then
\[
\sum_{n\leqslant x} a_nf(n) = A(x)f(x) -\int_{1}^x A(t)f'(t)\,dt.
\]
\end{lem}

\begin{proof}
    See, for example, Apostol \cite{Apostol} or Murty \cite{Murty}.
\end{proof}

\begin{lem}\label{number_of_zeroes}
Let $N(T)$ be the number of zeta-zeros in the rectangle $0<\sigma<1$, $0<t<T$. If $T\geqslant e$, then
\[
\left| N(T) - \frac{T}{2\pi} \log \frac{T}{2\pi e} - \frac{7}{8} \right| \leqslant 0.110 \log T + 0.290 \log \log T + 2.290 + \frac{25}{48\pi T}.
\]
\end{lem}

\begin{proof}
The proof is based on the results by Trudgian \cite{Trudgian}, Platt and Trudgian \cite{Platt_Trudgian}, Hasanalizade, Shen and Wong \cite{Has et al.}. See Goldštein and A.G. \cite[Lem. 7]{Grigutis} for the precise formulation.
\end{proof}

\begin{lem}\label{log_inequality}
Let 
\begin{align*}
a:=1+\frac{1}{2}W_0\left(-\frac{2}{e^2}\right)\approx0.7968,
\end{align*}
where $W_0(-2e^{-2})$ is a root $\neq-2$ of Lambert W function $-2e^{-2}=we^w$.
Then
\[
\log(1-x) > -2x,\,0<x<a.
\]
\end{lem}

\begin{proof}
Let $f(x)=\log(1-x)+2x$. One may check that $f(0)=f(a)=0$. Then the proof follows due to the first two derivatives 
\begin{align*}
f'(x)&=2+\frac{1}{x-1}>0,\,x<1/2,\\
f''(x)&=-\frac{1}{(x-1)^2}<0,\,x<1.
\end{align*}

\end{proof}

From now on, if not specified differently, we will assume that the functions $a(t)$ and $b(t)$ are both positive, with positive derivatives, when $t>t_0>0$, where $t_0$ is sufficiently large.

\begin{lem}\label{integral_inequality_kappa}
Let $\alpha >1$, $a(t)$ and $b(t)$ be the before defined functions independent of $u$ and such that $(t-\alpha)b(t)>a(t)$. Then, with $t> 2\alpha$, the following inequality holds
\[
\int_{\alpha}^{\infty} \frac{\log \frac{u}{2\pi}\,du}{a^2(t)+b^2(t)(u-t)^2} > \frac{\pi}{a(t)b(t)}\,\log \frac{t}{2\pi}-g(t),
\]
where
\begin{align}\label{g_f}
g(t)&=\frac{\log(t/2\pi)}{b^2(t)(t-\alpha)}\\
&+\frac{1}{t b^2(t)} \log\left(1+\frac{t^2 b^2(t)}{4a^2(t)}\right)+
\frac{1}{a^2(t)+b^2(t)(t-\alpha)^2}\left(\frac{\alpha}{t}\log\frac{\alpha}{t}-\frac{\alpha}{t}+\frac{1+\log 2}{2}\right).\nonumber
\end{align}
\end{lem}

\begin{proof}
Let us split the evaluated integral into two pieces
\begin{align}\label{two_P}
\int_{\alpha}^{\infty} \frac{\log \frac{u}{2\pi}\,du}{a^2(t)+b^2(t)(u-t)^2}=
\left(\int_{\alpha}^{t}+\int_{t}^{\infty}\right)\frac{\log \frac{u}{2\pi}\,du}{a^2(t)+b^2(t)(u-t)^2}=:P_1+P_2.
\end{align}

We now evaluate $P_1$. Noticing that
\[
\log \frac{u}{2\pi} = \log \frac{t}{2\pi} + \log \frac{u}{t},\,u>0,\,t>0,
\]
we have
\begin{align*}
&\int_{\alpha}^t \frac{\log \frac{u}{2\pi}\,du}{a^2(t)+b^2(t)(u-t)^2} =
\int_{\alpha}^t \frac{\log \frac{t}{2\pi}\,du}{a^2(t)+b^2(t)(u-t)^2} +
\int_{\alpha}^t \frac{ \log \frac{u}{t}\,du}{a^2(t)+b^2(t)(u-t)^2}\\
&=:I_1+I_2. 
\end{align*}
The first integral, under the change of variable $v=(u-t)b(t)/a(t)$, is
\begin{equation*}
\begin{split}
I_1=\int_{\alpha}^t \frac{\log \frac{t}{2\pi}\,du}{a^2(t)+b^2(t)(u-t)^2} 
&=\log \frac{t}{2\pi} \int_{\alpha}^t \frac{d\,(u-t)}{a^2(t)+b^2(t)(u-t)^2} \\
&= \log \frac{t}{2\pi} \cdot \frac{1}{a(t)b(t)}\int_{(\alpha-t)b(t)/a(t)}^0 \frac{dv}{1+v^2} \\
&= \log \frac{t}{2\pi} \cdot \frac{1}{a(t)b(t)} \arctan{\left((t-\alpha)\frac{b(t)}{a(t)}\right)},\,t>\alpha.
\end{split}
\end{equation*}
Applying the inequality $\arctan t>\pi/2-1/t$, $t>1$, see \cite[Lem. 9]{Grigutis}, we get
\begin{align*}
I_1>\frac{\pi/2\cdot \log (t/2\pi)}{a(t)b(t)}-\frac{\log(t/2\pi)}{b^2(t)(t-\alpha)},\,(t-\alpha)b(t)>a(t),\,t>\alpha+1.
\end{align*}

The second integral $I_2$, after the change of variable $x=t-u$, splits into two parts
\begin{equation*}
\begin{split}
&\int_{\alpha}^t \frac{\log\frac{u}{t}\,du}{a^2(t)+b^2(t)(u-t)^2}
=\int_{0}^{t-\alpha} \frac{ \log (1-\frac{x}{t})\, dx}{a^2(t)+b^2(t)x^2}\\
&= \int_0^{t/2}\frac{ \log (1-\frac{x}{t})\,dx}{a^2(t)+b^2(t)x^2} +\int_{t/2}^{t-\alpha}\frac{ \log (1-\frac{x}{t})\,dx}{a^2(t)+b^2(t)x^2} \eqqcolon \tilde{I}_1+\tilde{I}_2,\,t>2\alpha.
\end{split}
\end{equation*}
Let us evaluate $\tilde{I}_1$. Using Lemma \ref{log_inequality}
\[
\log \left(1-\frac{x}{t}\right) \geqslant -2\cdot\frac{x}{t},\,0<\frac{x}{t}<\frac{1}{2}
\]
we get
\begin{equation*}
\begin{split}
\tilde{I}_1 &=\int_0^{t/2}\frac{ \log (1-\frac{x}{t})\,dx}{a^2(t)+b^2(t)x^2} \geqslant -\int_0^{t/2}\frac{ 2x/t\,dx}{a^2(t)+b^2(t)x^2} = -\frac{1}{t}\int_0^{t/2} \frac{dx^2}{a^2(t)+b^2(t)x^2} \\
&=-\frac{1}{t b^2(t)}\eval{\log\left(a^2(t)+b^2(t)x^2\right)}_{0}^{t/2}
= -\frac{1}{t b^2(t)} \log\left(\frac{4a^2(t)+b^2(t)t^2}{4a^2(t)}\right),\,t>2\alpha.
\end{split}
\end{equation*}

Evaluating $\tilde{I}_2$ we do the change of variable $u=1-x/t$ and obtain
\begin{equation*}
\begin{split}
\tilde{I}_2&=\int_{t/2}^{t-\alpha}\frac{ \log (1-\frac{x}{t})\,dx}{a^2(t)+b^2(t)x^2}
\geqslant \frac{1}{a^2(t)+b^2(t)(t-\alpha)^2} \int_{t/2}^{t-\alpha}\log \left(1-\frac{x}{t}\right) \,dx \\
&=-\frac{t}{a^2(t)+b^2(t)(t-\alpha)^2}\int_{1/2}^{\alpha/t} \log(u)\,du \\
&=-\frac{t}{a^2(t)+b^2(t)(t-\alpha)^2} \left(\frac{\alpha}{t}\log \frac{\alpha}{t}-\frac{\alpha}{t}+\frac{1+\log 2}{2}\right),\,t>2\alpha.
\end{split}
\end{equation*}

The estimate of $P_2$ from \eqref{two_P} (see \cite[p. 836]{Grigutis}) is
\begin{align*}
P_2=\int_{t}^{\infty}\frac{\log \frac{u}{2\pi}\,du}{a^2(t)+b^2(t)(u-t)^2}>\frac{\pi/2}{a(t)b(t)}\log\frac{t}{2\pi},\,t>0.
\end{align*}

The proof follows combining the obtained inequalities of $I_1$, $\tilde{I}_1$, $\tilde{I}_2$, and $P_2$.
\end{proof}

Lemma \ref{integral_inequality_kappa} implies the following corollary.
\begin{corollary}\label{kappa_value}
If, in Lemma \ref{integral_inequality_kappa}, we set $a(t)=b(t)=\sqrt{\log t}$ and $\alpha=\gamma_1=14.134725\ldots$, where $\gamma_1$ is the lowest non-trivial zeta-zero $\zeta(1/2+i\gamma_1)=0$, we have that 
\begin{align*}
g(t)\sim\frac{2}{t},\,t\to\infty
\end{align*}
and $|g(t)|<0.0879\ldots$ as $t>2\alpha=28.2695\ldots$
\end{corollary}

\begin{lem}\label{integral_inequality2}
Suppose that $a(t)$ and $b(t)$ are positive, continuous, increasing functions, independent of $u$, and such that $tb(t)>a(t)$, when $t > 1$. Then
\[
\frac{1}{4tb^2(t)}\log\left(\frac{t}{2\pi}\right) -\frac{\alpha}{t^2b^2(t)} \log\left(\frac{\alpha}{2\pi}\right) < \int_{\alpha}^{\infty} \frac{\log \frac{u}{2\pi}\,du}{a^2(t)+b^2(t)(u+t)^2},\,t>\alpha>1.
\]
\end{lem}

\begin{proof}
Due to the functions $a(t)$ and $b(t)$ independence of $u$, the proof is practically identical to the first part of the proof given in \cite[Lem. 11]{Grigutis}
\end{proof}
In the further estimates, numbers are mostly rounded up to two or three decimal places.

\begin{lem}\label{sum_inequality1}
Let $\gamma$ be the imaginary part of a non-trivial zeta-zero. Let $a(t)$ and $b(t)$ be positive, continuous, increasing functions, independent of $u$ and $\gamma_1 = 14.134\ldots$, where $\zeta(1/2+i\gamma_1)=0$ and $\gamma_1$ is the lowest non-trivial zeta-zero. If $t> \gamma_1$, then
\begin{equation*}
\begin{split}
\sum_{\gamma>0}\frac{1}{a^2(t)+b^2(t)(t-\gamma)^2} 
&> \frac{1}{2\pi} \int_{\gamma_1}^{\infty} \frac{\log \frac{u}{2\pi}\,du}{a^2(t)+b^2(t)(u-t)^2} \\
&-\frac{0.22\log t+0.58\log\log t+4.58}{a^2(t)}-\frac{0.166}{ta^2(t)}\left(1+\frac{2.411 a(t)}{b(t)}\right).
\end{split}
\end{equation*}
\end{lem}

\begin{proof}
See \cite[Lem. 13]{Grigutis}.
\end{proof}

\begin{lem}\label{sum_inequality2}
Let $\gamma$ be the imaginary part of a non-trivial zero of $\zeta(s)$. Let $a(t)$ and $b(t)$ be positive, continuous, increasing functions, independent of $u$ and $\gamma_1 = 14.134725\ldots$, where $\zeta(1/2+i\gamma_1)=0$ and $\gamma_1$ is the lowest non-trivial zeta-zero. If $t>\gamma_1$, then
\begin{equation*}
\begin{split}
\sum_{\gamma>0}\frac{1}{a^2(t)+b^2(t)(t+\gamma)^2} &> \frac{1}{2\pi} \int_{\gamma_1}^{\infty} \frac{\log \frac{u}{2\pi}\,du}{a^2(t)+b^2(t)(u+t)^2} \\
&-\frac{3.811}{a^2(t)+b^2(t)(\gamma_1+t)^2} -\frac{0.045}{a(t)b(t)}.
\end{split}
\end{equation*}
\end{lem}

\begin{proof}
See \cite[Lem. 13]{Grigutis}.
\end{proof}

As mentioned, these days RH is verified numerically within the rectangle $0<\sigma<1$, $0<t< 3\times 10^ {12}$, see Platt and Trudgian \cite{Platt}. If we assume that there are infinitely many zeta-zeros lying off the critical line when $t>3\times10^{12}$, then the following statement holds.

\begin{lem}\label{sum_growth}
Assume that there are infinitely many non-trivial zeta-zeros $\tilde{\beta}_k+i\tilde{\gamma}_k$, $k=1,2,\ldots$, with $1/2<\tilde{\beta}_k<1$ and $\tilde{\gamma}_k > 3\times 10^{12}$. Then
\[
\sum_{\tilde{\gamma}_k >0} \frac{1}{(t+\tilde{\gamma}_k)^2}< \frac{1}{\pi}\cdot\frac{1+\log(t+\tilde{\gamma}_1)}{t+\tilde{\gamma}_1},\,t>0,\,\tilde{\gamma}_1>3\times 10^{12}.
\]
\end{lem}

\begin{proof}
    Applying Lemma \ref{partial_summation} and using the upper bound of $N(u)$ from Lemma \ref{number_of_zeroes} we get
    \[
    \sum_{\tilde{\gamma}_k >0} \frac{1}{(t+\tilde{\gamma}_k)^2}< -\int_{\tilde{\gamma}_1}^{\infty}N(u)f'(u)\,du < -\int_{\tilde{\gamma}_1}^{\infty}N_{up}(u)f'(u)\,du=:\tilde{I},
    \]
    where $f(u)=1/(t+u)^2$ and $N_{up}(u)$ is the upper bound of $N(u)$. Then,
\begin{equation*}
\begin{split}
\tilde{I}= \int_{\tilde{\gamma}_1}^{\infty} \frac{2N_{up}(u)\,du}{(t+u)^3}
=2\int_{\tilde{\gamma}_1}^{\infty} \frac{\frac{u}{2\pi}\log \frac{u}{2\pi e}+0.11\log u+0.29\log\log u+3.165+\frac{25}{48\pi u}}{(t+u)^3}\,du.
\end{split}
\end{equation*}
It is true that 
\begin{align*}
\frac{u}{2\pi}\log \frac{u}{2\pi e}+0.11\log u+0.29\log\log u+3.165+\frac{25}{48\pi u}<\frac{u\log u}{2\pi},\,u>8.032.
\end{align*}
Thus,
\begin{align*}
\tilde{I}&<\frac{1}{\pi}\int_{\tilde{\gamma}_1}^{\infty}\frac{u\log u\,du}{(t+u)^3}<
\frac{1}{\pi}\int_{\tilde{\gamma}_1}^{\infty}\frac{(t+u)\log (t+u)\,du}{(t+u)^3}\\
&=\frac{1}{\pi}\cdot\frac{1+\log(t+\tilde{\gamma}_1)}{t+\tilde{\gamma}_1},\,t>0,\,\tilde{\gamma}_1>3\times10^{12}.
\end{align*}

\end{proof}

\section{Proof of the main theorem}\label{sec3} 

\begin{customproof} 
    Notice that
\[
\mathfrak{R} \sum_{\rho = 1/2 + i \gamma} \frac{1}{s - \rho} = \sum_{\rho = 1/2 + i \gamma} \frac{\sigma-1/2}{(\sigma-1/2)^2+(t-\gamma)^2}.
\]
Since $\zeta(\rho)=\zeta(\overline{\rho})=0$ if $\rho=1/2+i\gamma$ is zeta-zero, we have
\begin{equation*}
\begin{split}
&\sum_{\rho = 1/2 + i \gamma} \frac{\sigma-1/2}{(\sigma-1/2)^2+(t-\gamma)^2}\\
&= \sum_{\gamma>0} \frac{\sigma-1/2}{(\sigma-1/2)^2+(t-\gamma)^2} + \sum_{\gamma>0} \frac{\sigma-1/2}{(\sigma-1/2)^2+(t+\gamma)^2} \\
&= \frac{1}{\sigma-1/2} \left(\sum_{\gamma>0} \frac{(\sigma-1/2)^2}{(\sigma-1/2)^2+(t-\gamma)^2} + \sum_{\gamma>0} \frac{(\sigma-1/2)^2}{(\sigma-1/2)^2+(t+\gamma)^2}\right) \\
&\eqqcolon \frac{1}{\sigma-1/2} (S_1+S_2).
\end{split}
\end{equation*}
    Let us evaluate each sum $S_1,\,S_2$ separately, assuming that
    \[
    \frac{1}{2}+\frac{1}{\sqrt{\log t}} < \sigma < 1 \quad \Rightarrow \quad \frac{1}{\log t} < \left(\sigma-\frac{1}{2}\right)^2<\frac{1}{4},
    \]
    when $t$ is sufficiently large. For the first sum, we have
\begin{equation*}
\begin{split}
S_1 &= \sum_{\gamma>0} \frac{(\sigma-1/2)^2}{(\sigma-1/2)^2+(t-\gamma)^2}\\
&>\frac{1}{\log t} \sum_{\gamma>0} \frac{1}{(\sigma-1/2)^2+(t-\gamma)^2}
>\frac{1}{\log t}\sum_{\gamma>0} \frac{1}{1+(t-\gamma)^2}.
\end{split}
\end{equation*}
According to Lemma \ref{partial_summation} and bound for number of zeta-zeros on the critical line \eqref{c_inequality} we get
\[
S_1 > -\int_{\gamma_1}^{\infty} N_{1/2}(u)f'(u)\,du \geqslant -c\int_{\gamma_1}^{\infty} N(u)f'(u)\,du,
\]
where $f(u) = 1/\log t/(1+(t-\gamma)^2)$. Then, applying Lemmas \ref{integral_inequality_kappa} and \ref{sum_inequality1} with $a(t)=b(t)=\sqrt{\log t}$ and $\alpha = \gamma_1$, we obtain
\begin{equation*}
\begin{split}
\frac{S_1}{c} >& -\int_{\gamma_1}^{\infty} N(u)f'(u)\,du \\
>&
\frac{1}{2\pi} \int_{\gamma_1}^{\infty} \frac{\log \frac{u}{2\pi}\,du}{\log t+\log t\,(u-t)^2}
-\frac{0.22\log t+0.58\log\log t+4.58}{\log t}-\frac{0.166}{t\log t}\cdot3.411 \\
>&\frac{1}{2\pi} \left(\frac{\pi}{\log t}\log\left(\frac{t}{2\pi}\right) - g(t)\right)
-\frac{0.22\log t+0.58\log\log t+4.58}{\log t}-\frac{0.566}{t\log t} \\
=& 0.28
-\frac{\log 2\pi}{2\pi\log t}-\frac{g(t)}{2\pi}-\frac{0.58 \log\log t}{\log t}
-\frac{4.58}{\log t}-\frac{0.566}{t\log t},
\end{split}
\end{equation*}
where the function $g(t)$ is given in \eqref{g_f} and $t$ is sufficiently large. By similar arguments, we evaluate the second sum
\[
S_2 = \sum_{\gamma>0} \frac{(\sigma-1/2)^2}{(\sigma-1/2)^2+(t+\gamma)^2} > \frac{1}{\log t}\sum_{\gamma>0} \frac{1}{1+(t+\gamma)^2},
\]
when $t$ is sufficiently large. Denoting $f(u) = 1/\log t/(1+(t+\gamma)^2)$ we get
\[
S_2 > -\int_{\gamma_1}^{\infty} N_{1/2}(u)f'(u)\,du \geqslant -c\int_{\gamma_1}^{\infty} N(u)f'(u)\,du.
\]
Then, as before, applying Lemmas \ref{integral_inequality2} and \ref{sum_inequality2} for the sufficiently large $t$ we obtain the lower bound:
\begin{equation*}
\begin{split}
\frac{S_2}{c} >& -\int_{\gamma_1}^{\infty} N(u)f'(u)\,du \\
>& \frac{1}{2\pi} \int_{\gamma_1}^{\infty} \frac{\log \frac{u}{2\pi}\,du}{\log t+\log t(u+t)^2}
-\frac{3.811}{\log t+\log t(\gamma_1+t)^2} -\frac{0.045}{\log t} \\
>& \frac{1}{4t\log t}\log\left(\frac{t}{2\pi}\right) -\frac{\gamma_1}{t^2\log t} \log\left(\frac{\gamma_1}{2\pi}\right) -\frac{3.811}{\log t+\log t(\gamma_1+t)^2} -\frac{0.045}{\log t}.
\end{split}
\end{equation*}
Gathering all of the previous estimates together, we get
\[
\mathfrak{R} \sum_{\rho = 1/2 + i \gamma} \frac{1}{s - \rho} > \frac{c\left(0.28-\varepsilon(t)\right)}{\sigma-1/2},
\]
where
\begin{equation}\label{epsilon}
\begin{split}
\varepsilon(t)&=\frac{\log 2\pi}{2\pi\log t}+\frac{0.58 \log\log t}{\log t}
+\frac{4.58}{\log t}+\frac{0.566}{t\log t}\\
&-\frac{1}{4t\log t}\log\left(\frac{t}{2\pi}\right) +\frac{\gamma_1}{t^2\log t}
\log\left(\frac{\gamma_1}{2\pi}\right) +\frac{3.811}{\left(1+t(\gamma_1+t)^2\right)\log t}+\frac{0.045}{\log t}\\
&+\frac{1}{2\pi}\left(\frac{\log \frac{t}{2\pi}}{(t-\gamma_1)\log t}+\frac{1}{t\log t}\log\left(1+\frac{t^2}{4}\right)+\frac{\frac{\gamma_1}{t}\log\frac{\gamma_1}{t}-\frac{\gamma_1}{t}+\frac{1+\log t}{2}}{(1+(t-\gamma_1)^2)\log t}\right)\\
&=O\left(\frac{\log\log t}{\log t}\right),\,t\to\infty.
\end{split}
\end{equation}
Using Mathematica \cite{Mathematica}, we verify that
\[
0.28-\varepsilon(t) >0, \text{ when } t > 3.11\times 10^{10}.
\]
\end{customproof}

\section{The positivity region of $\mathfrak{R}\xi '/\xi(s)$ if there are zeros off the critical line}\label{sec4}

In this section, in a similar manner as \cite[Sec. 4]{Grigutis}, we study the positivity region of $\mathfrak{R}\xi '/\xi(s)$ assuming that the Riemann hypothesis fails. There are three possible scenarios:
\begin{enumerate}[I.]
    \item There is only one zero in the region $1/2<\sigma<1$, $t>0$.
    \item There is a finite number $n\geqslant 2$ of zeros off the critical line.
    \item There are infinitely many zeros off the critical line.
\end{enumerate}

\RNum{1}. Assume that there exists a single point $\tilde{\rho}=\tilde{\beta}+i\tilde{\gamma}$ with $1/2<\tilde{\beta}<1$, $\tilde{\gamma} >0$, such that $\zeta(\tilde{\rho})=0$. Then, by Theorem \ref{main_theorem}, assuming that $c$ is sufficiently close to $1$,
\begin{equation*}
\begin{split}
\mathfrak{R}\frac{\xi'}{\xi}(s) =& \frac{1}{\sigma-1/2}\sum_{\rho=1/2+i\gamma} \frac{(\sigma-1/2)^2}{(\sigma-1/2)^2+(t-\gamma)^2} \\
+& \frac{\sigma - \tilde{\beta}}{(\sigma - \tilde{\beta})^2 + (t - \tilde{\gamma})^2} + \frac{\sigma - \tilde{\beta}}{(\sigma - \tilde{\beta})^2 + (t + \tilde{\gamma})^2} \\
+& \frac{\sigma - (1 - \tilde{\beta})}{(\sigma - (1 - \tilde{\beta}))^2 + (t - \tilde{\gamma})^2} + \frac{\sigma - (1 - \tilde{\beta})}{(\sigma - (1 - \tilde{\beta}))^2 + (t + \tilde{\gamma})^2} \\
>& \frac{0.28}{\sigma -1/2} +\frac{\sigma - \tilde{\beta}}{(\sigma - \tilde{\beta})^2 + (t - \tilde{\gamma})^2} +O\left(\frac{\log \log t}{\log t}\right) >0,
\end{split}
\end{equation*}
if
\begin{equation}\label{single_zero}
    (\sigma, t) \in \left\{\frac{\sigma - \tilde{\beta}}{(\sigma - \tilde{\beta})^2 + (t - \tilde{\gamma})^2} >-\frac{0.28}{\sigma -1/2} \right\},
\end{equation}
if $t$ is sufficiently large so that $\log \log t/\log t$ is negligible. Figure \ref{single_zero_figure} demonstrates the obtained hypothetical region \eqref{single_zero} with some chosen point $\tilde{\rho}$. Notice that if $t$ is assumed to be so large that $\log \log t/\log t$ is negligible, then $1/\sqrt{\log t}$ in $1/2+1/\sqrt{\log t}<\sigma<1$ is also negligible.

\RNum{2}. Assume that there is a finite number $n\geqslant2$ of points $\tilde{\rho}_k=\tilde{\beta}_k+i\tilde{\gamma}_k$, $k=1,\,2,\, \ldots,\, n$, with $1/2<\tilde{\beta}_k<1$, $\tilde{\gamma}_k > 0$, such that $\zeta(\tilde{\rho}_k)=0$. Then, by the same arguments as before,
\[
\mathfrak{R}\frac{\xi'}{\xi}(s) > \frac{0.28c}{\sigma -1/2} +\sum_{k=1}^n\frac{\sigma - \tilde{\beta}_k}{(\sigma - \tilde{\beta}_k)^2 + (t - \tilde{\gamma}_k)^2} > 0,
\]
if
\begin{equation}\label{finite_zeros}
    (\sigma, t) \in \left\{\sum_{k=1}^n\frac{\sigma - \tilde{\beta}_k}{(\sigma - \tilde{\beta}_k)^2 + (t - \tilde{\gamma}_k)^2} >-\frac{0.28c}{\sigma -1/2} \right\}
\end{equation}
and $t$ is sufficiently large. Figure \ref{finite_zero_figure} shows the hypothetical region \eqref{finite_zeros} with some chosen points and assuming that $c$ is sufficiently close to $1$.

\begin{figure}[H]
    \centering
    \begin{minipage}{0.45\textwidth}
        \centering
        \includegraphics[scale=0.6]{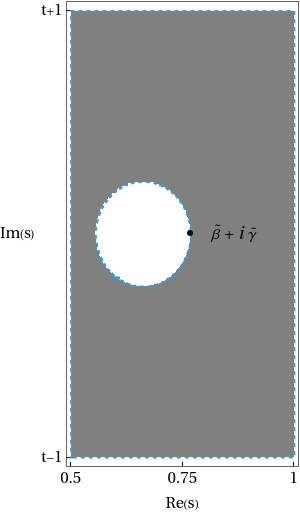}
        \caption{The gray region shows where the inequality \eqref{single_zero} holds, assuming a single hypothetical zero off the critical line.}
        \label{single_zero_figure}
    \end{minipage}%
    \hfill
    \begin{minipage}{0.45\textwidth}
        \centering
        \includegraphics[scale=0.6]{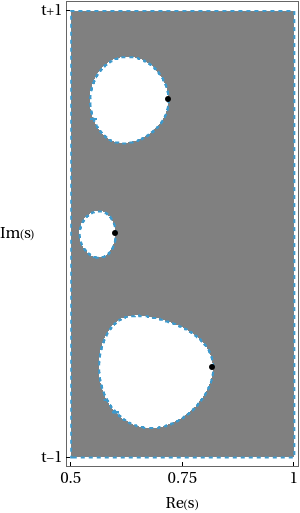}
        \caption{The gray region shows where the inequality \eqref{finite_zeros} holds, assuming several of the hypothetical zeros off the cri\-ti\-cal line.}
        \label{finite_zero_figure}
    \end{minipage}
\end{figure}

\RNum{3}. Assume that there are infinitely many points $\tilde{\rho}_k=\tilde{\beta}_k+i\tilde{\gamma}_k$, with $1/2<\tilde{\beta}_k<1$, $\tilde{\gamma}_k > 0$, such that $\zeta(\tilde{\rho}_k)=0$. Then, by the same argument as before,
\[
\mathfrak{R}\frac{\xi'}{\xi}(s) > \frac{0.28c}{\sigma -1/2} +\sum_{\tilde{\gamma}_k >0}\frac{\sigma - \tilde{\beta}_k}{(\sigma - \tilde{\beta}_k)^2 + (t - \tilde{\gamma}_k)^2} -\frac{1}{2}\sum_{\tilde{\gamma}_k >0} \frac{1}{(t+\tilde{\gamma}_k)^2} > 0,
\]
if
\begin{equation}\label{infinite_zeros}
    (\sigma, t) \in \left\{\sum_{\tilde{\gamma}_k >0}\frac{\sigma - \tilde{\beta}_k}{(\sigma - \tilde{\beta}_k)^2 + (t - \tilde{\gamma}_k)^2} >-\frac{0.28c}{\sigma-1/2} \right\},
\end{equation}
when $t$ is sufficiently large. We note that by Lemma \ref{sum_growth}
\[
\frac{1}{2}\sum_{\tilde{\gamma}_k >0} \frac{1}{(t+\tilde{\gamma}_k)^2}= O\left(\frac{\log t}{t}\right),\quad t\to \infty.
\]

\begin{remark}
Suppose the hypothetical zeros off the critical line are sufficiently close to one another. In that case, the regions of negativity they create might merge; however, it won't affect the area of positivity far from the mentioned hypothetical zeros. An example of this effect is shown in Figure \ref{merging_regions}. 
\end{remark}

\begin{figure}[H]
\centering
\includegraphics[scale=0.7]{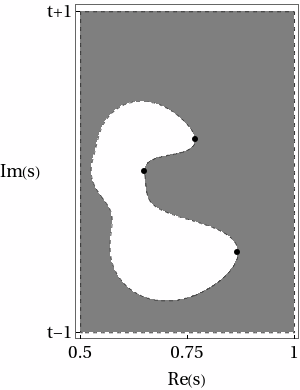}
\caption{The gray region shows where the inequality \eqref{finite_zeros} holds, assuming a small number of hypothetical zeta-zeros off the critical line that are located closely.}
\label{merging_regions}
\end{figure}

\end{document}